\documentclass[11pt]{amsart}

\usepackage[latin1]{inputenc}
\usepackage{t1enc}
\usepackage{amsmath, amssymb, amsfonts, amsthm, amsopn,epsfig}
\usepackage[none]{hyphenat}
\usepackage{tikz}
\usetikzlibrary{positioning,arrows,shapes,decorations.markings,decorations.pathreplacing,matrix,patterns}
\tikzstyle{vertex}=[circle,draw=black,fill=black,inner sep=0,minimum size=2pt,text=white,font=\footnotesize]
\usepackage{float}
\usepackage{graphicx}
\usepackage{caption}
\usepackage[toc,page]{appendix}
\usepackage{epstopdf}
\usepackage{etoolbox}
\usepackage[colorlinks=true]{hyperref}
\usepackage{dsfont}
\hypersetup{
	colorlinks=true,
	linkcolor=blue,
	filecolor=magenta,      
	urlcolor=cyan,
	citecolor=blue
}
\usepackage{cleveref}
\usepackage[top=25mm, bottom=25mm, left=25mm, right=25mm]{geometry}
\usepackage{comment}
\usepackage{algorithm2e}
\SetKwComment{Comment}{/* }{ */}

\theoremstyle{plain}
\newtheorem{theorem}{Theorem}[section]

\newtheorem{corollary}[theorem]{Corollary}
\newtheorem{claim}[theorem]{Claim}
\newtheorem{lemma}[theorem]{Lemma}
\newtheorem{conjecture}[theorem]{Conjecture}
\newtheorem*{conjecture*}{Conjecture}
\newtheorem*{problem*}{Problem}
\newtheorem{problem}[theorem]{Problem}

\theoremstyle{definition}
\newtheorem{definition}{Definition}

\DeclareMathOperator{\tr}{tr}
\DeclareMathOperator{\row}{row}
\DeclareMathOperator{\col}{col}
\DeclareMathOperator{\disc}{disc}
\DeclareMathOperator{\herdisc}{herdisc}
\DeclareMathOperator{\dgc}{dgc}
\DeclareMathOperator{\block}{bl}
\DeclareMathOperator{\POL}{POL}

\newcommand{\cF}{\mathcal{F}}

\begin{document}
	\sloppy 
	\title{Factorization norms and  Zarankiewicz problems}
\author{Istv\'an Tomon}
\address{Department of Mathematics and Mathematical Statistics, Ume\r{a} University, Ume\r{a}, 90736, Sweden}
\email{istvan.tomon@umu.se}
\thanks{IT is supported in part by the Swedish Research Council grant VR 2023-03375.}

\begin{abstract}
The $\gamma_2$-norm of Boolean matrices plays an important role in communication complexity and discrepancy theory. In this paper, we study combinatorial properties of this norm, and provide new applications, involving Zarankiewicz type problems.
\begin{itemize}
    \item We show that if $M$ is an $m\times n$ Boolean matrix such that $\gamma_2(M)<\gamma$ and $M$ contains no $t\times t$ all-ones submatrix, then $M$ contains $O_{\gamma,t}(m+n)$ one entries. In other words, graphs of bounded $\gamma_2$-norm are \textbf{degree bounded}.  This addresses a conjecture of Hambardzumyan,  Hatami, and Hatami for locally sparse matrices.
    \item   We prove that if $G$ is a $K_{t,t}$-free incidence graph of $n$ points and $n$ homothets of a polytope $P$ in $\mathbb{R}^d$, then the average degree of $G$ is $O_{d,P}(t(\log n)^{O(d)})$. This is sharp up the $O(.)$ notations. In particular, we prove a more general result on semilinear graphs, which greatly strengthens the work of Basit, Chernikov, Starchenko, Tao, and Tran.
\end{itemize}
\end{abstract}

%\keywords{Zarankiewicz problem, communication complexity, eigenvalues}
%\subjclass{}

\maketitle
    
%\address{Ume\r{a} University, \emph{e-mail}: \textbf{istvantomon@gmail.com}, Research supported in part by the Swedish Research Council grant VR 2023-03375.}
\section{Introduction}

Given a real matrix $M\in \mathbb{R}^{m\times n}$, the \emph{$\gamma_2$-norm} (or \emph{max-norm}) of $M$ is defined as
$$\gamma_2(M)=\min_{UV=M} ||U||_{\row} ||V||_{\col},$$
where $||U||_{\row}=||U||_{2\rightarrow\infty}$ is the maximum $\ell_2$-norm of the row vectors of $U$, and $||V||_{\col}=||V||_{1\rightarrow 2}$ is the maximum $\ell_2$-norm of the column vectors of $V$. This norm is equivalent to the \emph{nuclear norm}, defined as 
$$\nu(M)=\inf\left\{\sum_{i=1}^k |w_i|: \exists\mbox{ sign vectors }x_1,\dots,x_k,y_1,\dots,y_k,M=\sum_{i=1}^kw_i x_iy_i^T \right\}.$$
 The $\gamma_2$-norm has found profound applications in communication complexity and discrepancy theory. The aim of this paper is to study combinatorial properties of this norm, and to present new applications in extremal combinatorics and geometry.

\subsection{Matrices of bounded max-norm}

The central problem in communication complexity is to understand the structure of Boolean matrices (i.e. zero-one matrices) of certain complexity measures. For example, the celebrated log-rank conjecture of Lov\'asz and Saks \cite{LS93} is about decomposing low-rank matrices into all-zero and all-one rectangles. In this paper, our goal is to study Boolean matrices of small $\gamma_2$-norm, which also extend the family of small rank matrices. Indeed, a Boolean matrix of rank $r$ has $\gamma_2$-norm at most $\sqrt{r}$ \cite{LS}.

Every Boolean matrix of rank 1 has a very simple structure: it contains an all-ones submatrix, while all other entries are zero. The best known bounds on the log-rank conjecture \cite{L16,ST24} show that every rank $r$ Boolean matrix can be decomposed into $2^{O(\sqrt{r})}$ rank 1 Boolean matrices. This motivates the following analogous question for the $\gamma_2$-norm, proposed in \cite{HHH}. Is it true that every Boolean matrix of $\gamma_2$-norm at most $c$ is the linear combination of $O_c(1)$ Boolean matrices of $\gamma_2$-norm~1?  

The $\gamma_2$-norm of a Boolean matrix is 1 if and only if it is the blow-up of a permutation matrix. Call such a matrix as \emph{blocky matrix} (see Figure \ref{fig:blocky}), and let $\block(M)$ denote the minimum number of blocky-matrices, whose $\pm1$-linear combination is $M$.  As the $\gamma_2$-norm is subadditive, it follows that $\gamma_2(M)\leq \block(M)$. It is conjectured in \cite{HHH} that a weak qualitative converse of this  also holds. The following conjecture is equivalent to Conjecture III in \cite{HHH}, see \cite{HHH} for a detailed explanation.

\begin{figure}
\begin{center}
\begin{tikzpicture}
 \node at (0,0)  {$\begin{pmatrix}1 & 1 & 1 &  &  &  & \\
                                1 & 1 & 1 &  &  &  & \\
                                 &  &  & 1 & 1 &  & \\
                                 &  &  & 1 & 1 &  & \\
                                 &  &  & 1 & 1 &  & \\
                                 &  &  &  &  & 1 & \\
                                 &  &  &  &  & 1 & \\
                                 &  &  &  &  &  & 1
                                  \end{pmatrix}$};
                
\end{tikzpicture}
\caption{A \textbf{blocky matrix}, where blank entries denote zeros. Any row and column permutation is also a blocky matrix.}
\label{fig:blocky}
\end{center}
\end{figure}

\begin{conjecture}\label{conj:1}
   For every $\gamma>0$ there exists $b_{\gamma}$ such that every Boolean matrix $M$ with $\gamma_2(M)\leq \gamma$ satisfies $\block(M)\leq b_{\gamma}$.
\end{conjecture}

We prove some partial results towards Conjecture \ref{conj:1}, namely that it holds for locally sparse matrices, in which case we establish a much stronger result.

\begin{theorem}\label{thm:degree_bounded}
Let $t\geq 2$ be an integer and $\gamma>0$. Then there exists $d=d(\gamma,t)$ such that every $m\times n$ Boolean matrix $M$ with no $t\times t$ all-ones submatrix and $\gamma_2(M)\leq \gamma$ contains at most $d(m+n)$ one entries.
\end{theorem}

In the Concluding remarks, we discuss quantitative bounds on $d(\gamma,t)$. Theorem \ref{thm:degree_bounded} implies Conjecture \ref{conj:1} for Boolean matrices that avoid large all-ones submatrices, as we get the following corollary. Define the \emph{degeneracy} of a Boolean matrix $M$ as the smallest integer $d$ such that every submatrix of $M$ has a row or a column with at most $d$ one entries. In other words, if $M$ is the bi-adjacency matrix of a bipartite graph $G$, then the degeneracy of $M$ is equal to the degeneracy of $G$. Furthermore, say that a blocky matrix is \emph{thin} if every block of it has one row or one column.

\begin{corollary}\label{cor:equivalence}
    Let $\mathcal{M}$ be a family of Boolean matrices which contain no $t\times t$ all-ones submatrix. Then the following are equivalent.
    \begin{enumerate}
        \item $\exists \gamma$ s.t. $\forall M\in \mathcal{M}$: $\gamma_2(M)\leq \gamma$.
        \item  $\exists b$ s.t. $\forall M\in \mathcal{M}$: $M$ is the sum of at most $b$ thin blocky matrices.
        \item $\exists d$ s.t. $\forall M\in \mathcal{M}$: $M$ has degeneracy at most $d$.
    \end{enumerate}
\end{corollary}

 As further discussed in \cite{HHH}, Conjecture \ref{conj:1} has intricate connections to a celebrated result of Cohen \cite{Cohen} on idempotents, which was quantitatively strengthened by Green and Sanders \cite{GreenSanders} and Sanders \cite{Sanders}. Furthermore, we note that decompositions of matrices into linear combinations of blocky matrices is studied by Hambardzumyan, Hatami, and Hatami \cite{HHH} related to costs of certain communication protocols, while  Avraham and Yehudayoff \cite{AY} proves bounds on the minimal such decompositions for many natural families of matrices. 

Say that a Boolean matrix is \emph{four cycle-free} if it contains no $2\times 2$ all-ones submatrix. The main ingredient in the proof of Theorem \ref{thm:degree_bounded} is the following result, which shows that in case $M$ is four cycle-free, then the $\gamma_2$-norm of $M$ is essentially the square-root of its degeneracy. This result is quite powerful: while the $\gamma_2$-norm of classes of matrices is hard to estimate from a theoretical perspective, the degeneracy is a very easy parameter to handle. 

\begin{theorem}\label{thm:main}
Let $M$ be a four cycle-free Boolean matrix of degeneracy $d$. Then $$\gamma_2(M)=\Theta(\sqrt{d}).$$
\end{theorem}

\noindent
We discuss a number of applications of this theorem in the following sections.

\subsection{Communication complexity}

The $\gamma_2$-norm is an important tool in communication complexity, as demonstrated by a celebrated paper of Linial and Shraibman \cite{LS}. Given an $m\times n$ matrix $A$, let $\tilde{\gamma}_2(A)$ denote the minimum $\gamma_2$-norm of an $m\times n$ matrix $B$ that satisfies $|A(i,j)-B(i,j)|\leq 1/3$ for every entry $(i,j)\in [m]\times [n]$. Denoting by $R(A)$ the  public-coin
randomized communication complexity of $A$, and by $Q^*(A)$ the quantum communication complexity with
shared entanglement, the following inequality is proved in \cite{LS}:
$$\log \tilde{\gamma}_2(A) \lesssim Q^{*}(A)\leq R(A).$$
Linial and Shraibman \cite{LS} proposed the problem whether $\tilde{\gamma}_2(A)$ can be replaced with $\gamma_2(A)$ to get a similar lower bound for $R(A)$. However, this was recently disproved by  Cheung,  Hatami, Hosseini, and Shirley \cite{CHHS} in a strong sense, who constructed an $n\times n$  Boolean matrix $M$ such that $\gamma_2(M)\geq \Omega(n^{1/32})$ and $R(M)=O(\log n)$. Their main technical result is as follows.

Let $1\leq q\leq p$ be integers, and let $P=P(q,p)$ be the $qp\times qp$ Boolean matrix, whose rows and columns are indexed by the elements of $[q]\times \{0,\dots,p-1\}$, and its entries are given by $P[(x,x'),(y,y')]=1$ iff $xy+x'=y'$. Furthermore, let $P_p=P_p(q,p)$ be the matrix defined almost identically, but $P[(x,x'),(y,y')]=1$ iff $xy+x'=y'$ holds modulo $p$. In \cite{CHHS}, it is proved, by technical applications of Fourier analysis, that $\gamma_2(P_p)=\Omega(q^{1/8})$ if $q\leq \sqrt{p}$, and $\gamma_2(P)=\Omega(q^{1/8})$ if $q\leq p^{1/3}$.

However, note that $P$ and $P_p$ are the incidence matrices of points and lines, so they are four cycle-free. Therefore, Theorem \ref{thm:main} immediately implies the following improvements.

\begin{theorem}
Let $1\leq q\leq p-1$. Then  $\gamma_2(P_p)=\Theta(\sqrt{q})$ and $\gamma_2(P)=\Theta(\min\{\sqrt{q},p^{1/4}\})$.
\end{theorem}

\begin{proof}
 Given $x,x',y$, there is a unique $y'$ such that $xy+x'=y' \pmod{p}$, and also given $x,y,y'$, there is a a unique $x'$ such that $xy+x'=y' \pmod{p}$. Therefore, each row and column of $P_p$ contains $q$ one entries, so the degeneracy of $P_p$ is also $q$. By Theorem \ref{thm:main}, we get $\gamma_2(P_p)=\Theta(\sqrt{q})$.

Now let us consider $P$, and let us only prove the lower bound, we leave the upper bound as an exercise. We may assume that $q\leq \sqrt{p}$, as otherwise $P(\sqrt{p},p)$ is a submatrix of $P(q,p)$ and we use that the $\gamma_2$-norm of a submatrix is always at most the $\gamma_2$-norm of the matrix. Given $x\in [q]$, there are at least $qp/4$ solutions of $xy+x'=y'$ with $x,y\in [q]$ and $x',y'\in \{0,\dots,p-1\}$. Therefore, the number of one entries of $P$ is at least $q^2/4$, which means that the degeneracy of $P$ is at least $q/4$. Hence, by Theorem \ref{thm:main}, $\gamma_2(P_p)=\Omega(\sqrt{q})$. 
\end{proof}

Very recently, a similar result was obtained by Cheung, Hatami, Hosseini, Nikolov, Pitassi, and Shirley \cite{CHHNPS}, based on similar ideas. One of their main technical lemmas shows that if $M$ is a four-cycle free Boolean matrix, then $\gamma_2(M)\geq ||M||_2^2/\sqrt{2\Delta}$, where $\Delta$ is the maximum degree of the associated bipartite graph. This gives the same bound as Theorem \ref{thm:main} in case the bipartite graph is close to regular, otherwise Theorem \ref{thm:main} is stronger. In \cite{CHHNPS}, applications of the previous theorem are provided for bounding the \emph{deterministic communication protocol} with oracle access to Equality, denoted by $\texttt{D}^{\texttt{EQ}}$, of the \emph{Integer Inner Product} function $\texttt{IIP}_k^{(n)}$. This quantity is of great interest as it demonstrates large separation between the \emph{randomized communication protocol} and $\texttt{D}^{\texttt{EQ}}$.

\subsection{Zarankiewicz problem}

 Zarankiewicz's problem \cite{Zara} is a central question in extremal graph theory, asking for the maximum number of edges in a bipartite graph $G$ with vertex classes of size $m$ and $n$, which contains no copy of $K_{s,t}$, i.e. the complete bipartite graph with classes of size $s$ and $t$. To simplify notation, we focus on the most interesting case $m=n$ and $s=t$, for which the fundamental K\H{o}v\'ari-S\'os-Tur\'an theorem \cite{KST54} states that the maximum is $O_t(n^{2-1/t})$. On the other hand, the probabilistic deletion method shows the lower bound $\Omega_t(n^{2-2/(t+1)})$, see e.g. \cite{AS}. Therefore, the answer to Zarankiewicz's problem is of the order $n^{2-\Theta(1/t)}$.

 In the past two decades, Zarankiewicz type problems have been extensively studied in the setting in which we restrict the host graph $G$ to certain special graph families. Such results have important applications in incidence geometry \cite{FPSSZ,MST}, for example. A celebrated result of Fox, Pach, Suk, Scheffer, and Zahl \cite{FPSSZ} proves the following in this area. Say that a graph $G$ is \emph{semialgebraic} of description complexity $(d,D,s)$, if the vertices of $G$ are points in $\mathbb{R}^d$, and edges are pairs of points that satisfy a Boolean combination of $s$ polynomial inequalities of degree at most $D$ in $2d$ variables. In \cite{FPSSZ} it is proved that if $G$ is a $K_{t,t}$-free semialgebraic graph of description complexity $(d,D,s)$, then the number of edges of $G$ is at most $O_{d,s,D,t}(n^{2-2/(d+1)+o(1)})$. Hence, the exponent of $n$ only depends on the dimension $d$, and it does not depend on $t$. Qualitatively, the same phenomenon holds in   a much more general setting. If $\cF$ is a hereditary family of graphs which does not contain every bipartite graph, then there exists $c=c(\cF)>0$ such that every $K_{t,t}$-free $n$ vertex graph in $\cF$ has at most $O_{\cF,t}(n^{2-c})$ edges. See \cite{BBCD,GH,HMST}, where \cite{HMST} focuses on finding the best possible $c$ for families $\cF$ defined by a forbidden induced bipartite graph $H$. In case $c$ can be chosen to be 1, that is, if every $K_{t,t}$-free member of $\cF$ has average degree $O_{\cF,t}(1)$, then the family $\cF$ is called \emph{degree-bounded}. Such families are of great interest in structural graph theory \cite{BBCD,GH,HMST,SSS} and combinatorial geometry \cite{CH23,FP08,KS}. An immediate corollary of Theorem \ref{thm:degree_bounded} is the following.

 \begin{theorem}
 Let $\gamma>0$ and let $\cF$ be the family of graphs, whose adjacency matrix has $\gamma_2$-norm at most $\gamma$. Then $\cF$ is degree bounded.
 \end{theorem}

Finding the best possible exponent $c(\cF)$ for several geometrically defined graph families has also been a fruitful topic, see the recent survey of Smorodinsky \cite{Smor} for a detailed overview. For example, Keller and Smorodinsky \cite{KS} show that if $G$ is the incidence graph of $n$ points and $n$ pseudo-disks, and $G$ contains no $K_{t,t}$, then $G$ has at most $O_t(n)$ edges. Chan and Har-Peled \cite{CH23} proves the same result for incidence graphs of points and half-spaces in dimensions 2 and 3, and show that such results no longer holds for dimension $d\geq 5$. 

\begin{figure}
\begin{center}
%\begin{tikzpicture}
%    \input{polytopes.tex}
%\end{tikzpicture}
\includegraphics[scale=0.3333]{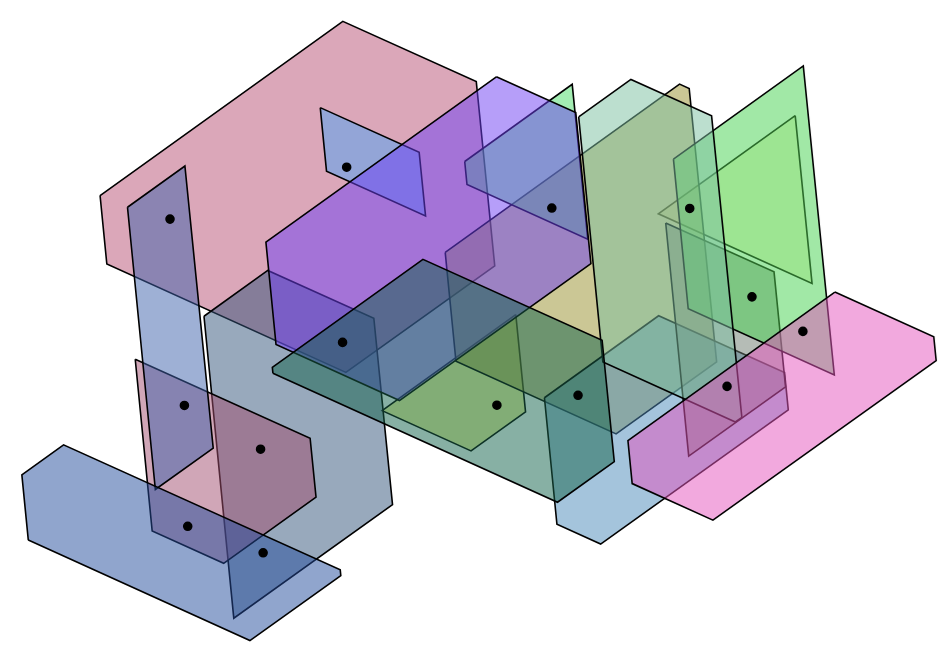}
\caption{A $K_{2,2}$-free configuration of points and elements of $\POL(\mathcal{H})$, where $\mathcal{H}$ contains 6 half-planes.}
\label{fig:polytope}
\end{center}
\end{figure}

In this paper, we study the following set of problems, first proposed by Basit, Chernikov, Starchenko, Tao and Tran \cite{BCSTT} and Tomon and Zakharov \cite{TZ}. In \cite{BCSTT}, it is proved that if $G$ is the incidence graph of $n$ points and $n$ axis-parallel boxes in $\mathbb{R}^d$, and $G$ is $K_{t,t}$-free, then it has at most $O_{d,t}(n(\log n)^{2d})$ edges. Subsequently, this bound was improved to $O_d(tn(\log/\log\log n)^{d-1})$ by Chan and Har-Peled \cite{CH23}, who also presented a matching lower bound construction (see also~\cite{T24}). 

More generally, in \cite{BCSTT} the following question is studied. Given a set $\mathcal{H}$ of $s$ half-spaces in $\mathbb{R}^d$, let $\POL(\mathcal{H})$ denote the set of polytopes that are intersections of translates of elements of $\mathcal{H}$, see Figure \ref{fig:polytope} for an illustration. In particular, if $P\in\POL(\mathcal{H})$, then  $\POL(\mathcal{H})$ contains all homothets of $P$ (but many other polytopes as well). In \cite{BCSTT}, it is proved that if $G$ is the $K_{t,t}$-free incidence graph of $n$ points and $n$ polytopes in $\POL(\mathcal{H})$, then $G$ has at most $O_{s,t}(n(\log n)^s)$ edges.  In \cite{CH23}, this is improved to $O_{s}(tn(\log n/\log\log n)^{\delta-1})$, where $\delta$ is the maximum size of a subset of $\mathcal{H}$ with no two half-spaces having parallel boundaries  (so $\delta\geq s/2$). The main result of this section greatly strengthens these results by showing that the exponent of the logarithm need not grow with $s$, it only depends on the dimension of the space.

\begin{theorem}\label{thm:zara1}
Let $\mathcal{H}$ be a set of $s$ half-spaces in $\mathbb{R}^d$, and let $Q$ be a set of $n$ polytopes, each of which is an intersection of translates of elements of $\mathcal{H}$. If $G$ is the incidence graph of a set of $n$ points and $Q$, and $G$ is $K_{t,t}$-free, then $G$ has at most $O_{s}(tn(\log n)^{O(d)})$ edges.
\end{theorem}

This theorem is sharp, even if we restrict $Q$ to be a family of translates of any fixed polytope $P$ with positive volume. Indeed, any incidence graph of points and boxes in $\mathbb{R}^{D}$ is an incidence graph of points and corners in $\mathbb{R}^{2D}$, where a \emph{corner} is a set of the form $C_t=\{x\in\mathbb{R}^{d}:\forall i, x(i)<t(i)\}$, $t\in\mathbb{R}^{d}$. But then, given a $K_{2,2}$-free configuration of $n$ points and $n$ boxes in $\mathbb{R}^{\lfloor d/2\rfloor}$ with $n(\log n)^{\lfloor d/2\rfloor -1-o(1)}$ incidences (which exists by the aforementioned construction of Chan and Har-Peled \cite{CH23}), we can transform it into a configuration of points and translates of $P$ with the same incidence graph.

The proof of Theorem \ref{thm:zara1} is based on studying the $\gamma_2$-norm of incidence matrices of points and polytopes in $\POL(H)$, and then using spectral methods to find large all-ones submatrices. This is vastly different from the approaches of \cite{BCSTT} and \cite{CH23}, and to the best of our knowledge, it is the first proof in the area that relies on linear algebraic techniques.

Theorem \ref{thm:zara1} is a special subcase of the following result about \emph{semilinear graphs}. Semilinear graphs form the subfamily of semialgebraic graphs, in which the defining polynomials are linear functions (so $D=1$ in the above definition). In \cite{BCSTT}, it is shown that if $G$ is an $n$ vertex semilinear graph of description complexity $(s,u)$, and $G$ is $K_{t,t}$-free, then it has at most $O_{t,s,u}(n(\log n)^{s})$ edges. We refer the reader to Section \ref{sect:zara} for formal definitions.  We show that, analogously to the case of semialgebraic graphs, the order of the function can be bounded by the dimension instead of the complexity.

\begin{theorem}\label{thm:zara2}
Let $G$ be an $n$ vertex graph, whose vertices are points in $\mathbb{R}^d$, and a pair of points form an edge if they satisfy a Boolean combination of $s$ linear inequalities in $2d$ variables. If $G$ contains no $K_{t,t}$, then $G$ has at most $O_{d,s}(tn(\log n)^{O(d)})$ edges.
\end{theorem}

If polynomials of degree at least $2$ are also permitted, similar results no longer hold. Indeed, the incidence graph of $n$ points and $n$ lines in $\mathbb{R}^2$ is semialgebraic of description complexity $(2,2,1)$, it is $K_{2,2}$-free, and it can have as many as $\Omega(n^{4/3})$ edges \cite{SzT}.  

%Our proofs of Theorems \ref{thm:zara1} and \ref{thm:zara2} are based on studying the $\gamma_2$-norm of incidence matrices of points and polytopes in $\mathbb{R}^d$. This is quite different from the approach of \cite{BCSTT}, which is based on divide-and-conquer type arguments. However, there is a high level similarity with the proof of Chan and Har-Peled \cite{CH23}. Their argument is based on decomposing the incidence graph into a small number of graphs, each of which is the vertex disjoint union of complete bipartite graphs. But this is equivalent to writing the incidence matrix as the sum of blocky matrices. 

\subsection{Discrepancy theory}

The $\gamma_2$-norm has important applications in discrepancy theory as well. Let $M$ be an $m\times n$ matrix, then the \emph{discrepancy} (also referred to as combinatorial discrepancy) of $M$ is defined as
$$\disc(M)=\min_{x\in \{-1,1\}^n} ||Mx||_{\infty}.$$
Here, $||.||_{\infty}$ is the maximum absolute value of the entries. Moreover, the \emph{hereditary discrepancy} of $M$ is defined as $\herdisc(M)=\max_{N\subset M}  \disc(N),$ where the maximum is taken over all submatrices $N$ of $M$.  If $\cF$ is set system on a ground set $X$, then $\disc(\cF)=\disc(M)$ and $\herdisc(\cF)=\herdisc(M)$, where $M$ is the incidence matrix of $\cF$ (with rows representing the sets). In combinatorial terms, the discrepancy of $\cF$ is the minimal $k$ for which there is a red-blue coloring of the elements of $X$ such that the numbers of red and blue elements in each set of $\cF$ differ by at most $k$.

Combinatorial discrepancy theory has its roots in the study of irregularities of distributions, and became a highly active area of research since the 80's \cite{BC}. It also found profound applications in computer science, see the book of Chazelle \cite{Ch} as a general reference. A classical result in the area is the Beck-Fiala theorem, which states that if $\cF$ is a set system such that each element of $X$ appears in at most $d$ sets, then $\disc(\cF)=O(d)$. The discrepancy of geometrically defined set systems is also extensively studied. Given a set of points $X$ in $\mathbb{R}^d$ and a collection $\mathcal{C}$ of geometric objects, one typically studies the discrepancy of the system $\cF=\{X\cap C:C\in\mathcal{C}\}$. Instances of these include when $\mathcal{C}$ is a collection of axis-parallel boxes \cite{Larsen,MNT,Nikolov}, lines \cite{ChL}, half-spaces \cite{half_space,CMS}, Euclidean balls \cite{Alex}, certain polytopes \cite{Beck,Nikolov}.

The following general inequality of Matou\v{s}ek, Nikolov, and Talwar \cite{MNT} establishes a sharp relation between the $\gamma_2$-norm and the hereditary discrepancy of arbitrary matrices:
    $$\Omega\left(\frac{\gamma_2(M)}{\log m}\right)\leq \herdisc(M)=O(\gamma_2(M)\sqrt{\log m}).$$
Combining this theorem with Theorem \ref{thm:main} immediately gives that if $M$ is a four-cycle free Boolean matrix, then $\herdisc(M)$ and $\sqrt{\dgc(M)}$ are equal up to logarithmic factors. For example, if $M$ is the incidence matrix of $n$ points and $m$ lines in the plane, then $M$ is four cycle-free and the Szemer\'edi-Trotter theorem \cite{SzT} implies that $\dgc(M)=O(n^{1/3})$. This bound is also the best possible, so we get close to optimal bounds on the discrepancy of geometric set systems generated by lines, recovering the results of \cite{ChL}.

\subsection*{Paper organization}
 In the next section, we present the main definitions and notions used throughout our paper. Then, in Section \ref{sect:4cycle}, we prove Theorems \ref{thm:degree_bounded}, \ref{thm:main} and Corollary \ref{cor:equivalence}. We continue with the proof of Theorems \ref{thm:zara1} and \ref{thm:zara2} in Section \ref{sect:zara}.

\section{Preliminaries}

In this section, we introduce the basic notation used throughout this paper and present some simple results.

\subsection{Combinatorics of matrices}

Given a Boolean matrix $M\in \{0,1\}^{m\times n}$, it naturally corresponds to the bipartite graph $G$ with vertex classes $[m]$ and $[n]$, where there is an edge between $i\in [m]$ and $j\in [n]$ if and only if $M(i,j)=1$. The matrix $M$ is the \emph{bi-adjacency} matrix of $G$. We adapt certain graph theoretic notations to Boolean matrices, e.g. the average degree of a matrix $M$ is the average degree of $G$, and a matrix is four cycle-free if it contains no $2\times 2$ all-ones submatrix. 

\begin{definition}[Degeneracy]
    Given a graph $G$ and a nonnegative integer $d$, $G$ is \emph{$d$-degenerate} if every subgraph of $G$ has a vertex of degree at most $d$. The \emph{degeneracy} of $G$ is the smallest $d$ such that $G$ is $d$-degenerate, and it is denoted by $\dgc(G)$. If $M$ is a Boolean matrix and $G$ is the bipartite graph with bi-adjacency matrix $M$, we define $\dgc(M)=\dgc(G)$.
\end{definition}

\subsection{Linear algebra notation}
Let $M$ be an $m\times n$ real matrix. The \emph{Schatten $p$-norm} of $M$ is defined as
$$||M||_p=\left(\sum_{i=1}^{\min\{m,n\}}\sigma_i^p\right)^{1/p},$$
where $\sigma_1,\dots,\sigma_{\min\{m,n\}}$ are the singular values of $M$. The \emph{trace-norm} of $M$ is the Schatten 1-norm, that is, $||M||_{\tr}=||M||_1$. 

Next, we discuss some basic operations between matrices. Let $M\in \mathbb{R}^{m\times n}$ and $M'\in \mathbb{R}^{m'\times n'}$.
\begin{itemize}
    \item (direct sum) $M\oplus M'$ is the $(m+m')\times (n+n')$ matrix $N$ defined as $N(i,j)=M(i,j)$ if $(i,j)\in[m]\times [n]$, $N(i+m,j+n)=M'(i,j)$ if $(i,j)\in [m']\times [n']$, and $N(i,j)=0$ for all unspecified entries.
    \item (Kronecker product/direct product)  $M\otimes M'$ is the $(mm')\times (nn')$ matrix $N$ defined as $N((i,i'),(j,j'))=M(i,j)M'(i',j')$ for $(i,j)\in [m]\times [n]$ and $(i',j')\in [m']\times [n']$.
    \item (Hadamard product) if $m=m'$ and $n=n'$, then $M\circ M'$ is the $m\times n$ matrix $N$ defined as $N(i,j)=M(i,j)M'(i,j)$.
\end{itemize}

\subsection{The max-norm}

In this section, we collect some basic properties of the $\gamma_2$-norm. We refer the reader to \cite{LSS} as a general reference.

\begin{definition}[$\gamma_2$-norm]
    Let $M$ be an $m\times n$ real matrix. The \emph{$\gamma_2$-norm} (or \emph{max-norm}) of $M$ is defined as
$$\gamma_2(M)=\min_{UV=M} ||U||_{\row} ||V||_{\col},$$
where $||U||_{\row}=||U||_{2\rightarrow\infty}$ is the maximum $\ell_2$-norm of the row vectors of $U$, and $||V||_{\col}=||V||_{1\rightarrow 2}$ is the maximum $\ell_2$-norm of the column vectors of $V$.
\end{definition} 

\noindent
Let $M\in \mathbb{R}^{m\times n}$ and let $N$ be a real matrix.
\begin{enumerate}
    \item If $c\in \mathbb{R}$, then $\gamma_2(cM)=|c|\gamma_2(M)$.
    \item (monotonicity) If $N$ is a submatrix of $M$, then $\gamma_2(N)\leq \gamma_2(M)$.
    \item (subadditivity) If $M$ and $N$ have the same size, then $\gamma_2(M+N)\leq \gamma_2(M)+\gamma_2(N)$.
    \item  $\gamma_2(M)=\max||M\circ (u v^T)||_{\tr},$ where the maximum is over all unit vectors $u\in \mathbb{R}^m$, $v\in \mathbb{R}^{n}$.
    \item  $\gamma_2(M)\geq \frac{1}{\sqrt{mn}}||M||_{\tr}.$
    \item  $\gamma_2(M\otimes N)=\gamma_2(M)\gamma_2(N)$.
    \item Duplicating rows or columns of $M$ does not change the $\gamma_2$-norm.
    \item $\gamma_2(M)\leq \min\{||M||_{\row},||M||_{\col}\}$
    \item  $\gamma_2(M\oplus N)=\max\{\gamma_2(M),\gamma_2(N)\}$.
    \item If $M$ is Boolean, then $\gamma_2(M)\leq \sqrt{\mbox{rank}(M)}$.
\end{enumerate}

We note that 5. follows from 4. by taking $u$ and $v$ be the normalized all-ones vectors. Moreover, 8. follows by setting $(U,V)=(I,M)$ or $(U,V)=(M,I)$ in the definition of the $\gamma_2$-norm.

\subsection{Blocky matrices}
In this section, we collect basic properties of the $\block(.)$ function.

\begin{definition}[Blocky matrix]
     A \emph{blocky matrix} is a Boolean matrix $M$ whose rows and columns can be partitioned into sets $A_0,A_1,\dots,A_k$ and $B_0,B_1,\dots,B_k$ for some $k\geq 0$ such that $M[A_i\times B_i]$ is the all-ones matrix for $i=1,\dots,k$, and $M[A_i\times B_j]=0$ for $i\neq j$ and $i=j=0$. We refer to the submatrices $M[A_i\times B_i]$ and rectangles $A_i\times B_i$ for $i=1,\dots,k$ as \emph{blocks}. Finally, say that a blocky matrix is \emph{thin} is for every $i=1,\dots,k$, either $|A_i|=1$ or $|B_i|=1$.
\end{definition}

\begin{definition}
Let $M$ be an integer matrix, then $\block(M)$  is the minimum $k$ for which there exist $k$ blocky matrices $B_1,\dots,B_k$ and $\varepsilon_1,\dots,\varepsilon_k\in\{-1,1\}$ such that $M=\sum_{i=1}^k \varepsilon_i B_i$. 
\end{definition}

\noindent
Let $M$ and $N$ be an integer matrices.
\begin{enumerate}
    \item (monotonicity) If $N$ is a submatrix of $M$, then $\block(N)\leq \block(M)$.
    \item (subadditivity) If $M$ and $N$ have the same size, then $\block(M+N)\leq \block(M)+\block(N)$.
    \item $\gamma_2(M)\leq \block(M)$.
    \item $\block(M\otimes N)\leq \block(M)\block(N)$.
    \item Duplicating rows or columns of $M$ does not change $\block(M)$.
    \item $\block(M\oplus N)=\max\{\block(M),\block(N)\}$.
\end{enumerate}

\noindent
Here, 4. follows by noting that the Kronecker product of blocky matrices is also a blocky matrix.

\subsection{Basic results}

In this section, we collect a few elementary results about the degeneracy, $\gamma_2$-norm, and $\block(.)$.

\begin{claim}\label{claim:dgc}
If $M$ is a Boolean matrix, then $\gamma_2(M)\leq 2\sqrt{\dgc(M)}$. 
\end{claim}

\begin{proof}
Let $d=\dgc(M)$, let $G$ be the underlying bipartite graph with vertex classes $A$ and $B$ (corresponding to rows and columns of $M$, respectively). Then there exists an ordering $<$ of $V(G)=A\cup B$ such that every $v\in V(G)$ has at most $d$ neighbours $<$-larger than $v$. Let $G_1$ be the subgraph of $G$ in which we keep those edge, whose $<$-smaller element is in $A$, and let $G_2$ be the rest of the edges. If $M_i$ is the bi-adjacency matrix of $G_i$ for $i=1,2$, then $M=M_1+M_2$, every row of $M_1$ has at most $d$ one entries, and every column of $M_2$ has at most $d$ one entries. Hence, by property 8., $\gamma_2(M_1)\leq ||M_1||_{\row}\leq\sqrt{d}$ and $\gamma_2(M_2)\leq ||M_2||_{\col}\leq\sqrt{d}$. Finally, by subadditivity, $\gamma_2(M)\leq \gamma_2(M_1)+\gamma_2(M_2)\leq 2\sqrt{d}$.
\end{proof}

\begin{claim}\label{claim:blocky}
   Let $M$ be a Boolean matrix, and let $k$ be the minimum number of thin blocky matrices, whose sum is $M$. Then
   $$\frac{\dgc(M)}{2}\leq k\leq 2\dgc(M).$$
\end{claim}

\begin{proof}
We first prove the upper bound. Let $d=\dgc(M)$, then by the previous proof, we can  write $M=M_1+M_2$, where every row of $M_1$ has at most $d$ one entries, and every column of $M_2$ has at most $d$ one entries. For $\ell=1,\dots,d$, let $B_{1,\ell}$ be the matrix, where $B_{1,\ell}(i,j)=1$ if $M_1(i,j)=1$ and $M_1(i,j)$ is the $\ell$-th one entry in the $i$-th row of $M_1$, otherwise let $B_{1,\ell}=0$. Then $M_1=B_{1,1}+\dots+B_{1,d}$ and $B_{1,\ell}$ is a thin blocky matrix. We define similarly the matrices $B_{2,1},\dots,B_{2,d}$ with respect to the columns of $M_2$. But then 
$M=\sum_{\ell=1}^d (B_{1,\ell}+B_{2,\ell})$, finishing the proof.

Now let us turn to the lower bound.  Let $B_1,\dots,B_k$ be thin blocky matrices, whose sum is $M$. Let $X$ be a set of rows, $Y$ be a set of columns. Note that $B_i[X\times Y]$ contains at most $|X|+|Y|-1$ one entries, hence $M[X\times Y]$ contains at most $k(|X|+|Y|-1)$ entries. Thus, assuming that $|X|\leq |Y|$, there is a column containing at most $k(|X|+|Y|-1)/|Y|<2k$ one entries. As this holds for every submatrix of $M$, we conclude that $\dgc(M)<2k$.
\end{proof}

%\begin{claim}\label{claim:blocky2}
%Let $M$ be an integer matrix. Then one can write $M=UV$ such that every entry of $U$ and $V$ is in $\{-1,0,1\}$, every row of $U$ has at most $\block(M)$ one entries, and every column of $V$ has at most $\block(M)$ one entries.
%\end{claim}

%\begin{proof}
%Let $m\times n$ be the size of $M$, let $k=\block(M)$, and write $M=\sum_{i=1}^k\varepsilon_i B_i$, where $B_i$ is a blocky matrix and $\varepsilon_i\in\{-1,1\}.$ Let $X_{i,j}\times Y_{i,j}$ be the blocks of $B_i$ for $i=1,\dots,k$ and $j=1,\dots,s_k$. Writing $r=s_1+\dots+s_k$, we define $U$ and $V$ to be $m\times r$ and $r\times n$ matrices, respectively.
%We set $U(a,b)=1$ if $b=s_1+\dots+s_{i-1}+j$ for some $x\in \{1,\dots,s_i\}$ and $a\in X_{i,j}$, otherwise 0, and $V(b,c)=\varepsilon_i$ if $b=s_1+\dots+s_{i-1}+j$ for some $x\in \{1,\dots,s_i\}$ and $c\in Y_{i,j}$. These $U$ and $V$ satisfy the required properties.
%\end{proof}

\section{Sparse matrices}\label{sect:4cycle}

In this section, we prove Theorems \ref{thm:degree_bounded}, \ref{thm:main} and Corollary \ref{cor:equivalence}. Most of this section is devoted to proving that four cycle-free matrices of average degree $d$ have $\gamma_2$-norm at least $\Omega(\sqrt{d})$. From this, Theorem \ref{thm:main} follows after a bit of work. Then, we show that Theorem \ref{thm:main} implies Theorem \ref{thm:degree_bounded}.

Let $M$ be a matrix of average degree at least $d$. The first step is to find a submatrix of average degree $\Omega(d)$ where either each row or each column contains $\Theta(d')$ entries for some $d'=\Omega(d)$. Unfortunately, it is a well known result of graph theory \cite{JS} that it is not always possible to find a submatrix in which this is true for both the rows and columns simultaneously, which would also make our proof significantly simpler.

\begin{lemma}\label{lemma:regularize}
Let $M$ be a Boolean matrix of average degree at least $d$. Then $M$ contains a submatrix $N$ of average degree $d'\geq d/3$ such that every row and column of $N$ contains at least $d'/2$ one entries, and either $||N||_{\row}^2\leq 6d'$ or $||N||_{\col}^2\leq 6d'$.
\end{lemma}

\begin{proof}
Let $G$ be the bipartite graph, whose bi-adjacency matrix is $M$, and let $A$ and $B$ be the vertex classes of $G$. Our task is to show that $G$ contains an induced subgraph $G'$ of average degree  $d'\geq d/3$ such that $G'$ has minimum degree at least $d'/2$, and every degree in one of the parts is at most $6d'$. 

Let $G_0$ be an induced subgraph of $G$ of maximum average degree, let $d_0$ be the average degree of $G_0$, then $d_0\geq d$. First, we note that $G_0$ has no vertex of degree less than $d_0/2$. Indeed, otherwise, if $v\in V(G_0)$ is such a vertex, then the average degree of $G_0-v$ (i.e., the graph we get by removing $v$) has average degree $2e(G_0-v)/(v(G_0)-1)>(2e(G_0)-d_0)/(v(G_0)-1)=d_0$.

Let $A_0\subset A,B_0\subset B$ be the vertex classes of $G_0$, and assume without loss of generality that $|A_0|\geq |B_0|$. Note that the number of edges of $G_0$ is $\frac{d_0}{2}(|A_0|+|B_0|)$. Let $C\subset A_0$ be the set of vertices of degree more than $2d_0$, then $|C|\leq |A_0|/2$. Indeed, otherwise, the number of edges of $G_0$ is at least $2d_0|C|>d_0|A_0|\geq\frac{d_0}{2}(|A_0|+|B_0|)$, contradiction. Let $A_1=A_0\setminus C$, and let $G_1$ be the subgraph of $G_0$ induced on $A_1\cup B_0$. The number of edges of $G_1$ is at least $d_0|A_1|/2\geq d_0 (|A_1|+|B_0|)/6$, so the average degree of $G_1$ is at least $d_0/3$. Let $G'$ be an induced subgraph of $G_1$ of maximum average degree, and let $d'$ be the average degree of $G'$. Then $d'\geq d_0/3$, and every vertex of $G'$ has degree at least $d'/2\geq d_0/6\geq d/6$. Furthermore, if $A'\subset A_1$ and $B'\subset B_0$ are the vertex classes of $G'$, then every degree in $A'$ is at most $2d_0\leq 6d'$. This finishes the proof.
\end{proof}

Now the idea of the proof is as follows. After passing to a submatrix $N$ which is close to regular from one side, say all columns have $\Theta(d')$ one entries, we use the fact that $$\gamma_2(N)\geq||N\circ (u v^T)||_{\tr}$$
for any choice of unit vectors $u$ and $v$ (of the appropriate dimension). We choose $u$ to be a vector, whose entries are based on the degree distribution of the rows, and choose $v$ to be the normalized all-ones vector. Let $A=N\circ (u v^T)$, then we inspect the matrices $B=AA^T$ and $B^2$. With the help of the Cauchy interlacing theorem, we show that the singular values of $A$ follow a certain distribution,  and thus find a lower bound for $||A||_{\tr}$.

\begin{lemma}\label{lemma:C4-free}
Let $M$ be a four cycle-free Boolean matrix of average degree at least $d$. Then $$\gamma_2(M)=\Omega(\sqrt{d}).$$
\end{lemma}
\begin{proof}
Let $N$ be a submatrix of $M$ satisfying the outcome of Lemma \ref{lemma:regularize}. Let $d_0$ be the average degree of $N$, then $d_0\geq d/3$, every row and column of $N$ contains at least $d_0/2$ one entries, and $||N||_{\row}^2\leq 6d_0$ or $||N||_{\col}^2\leq 6d_0$. Without loss of generality, we assume that $||N||_{\col}^2\leq 6d_0$. In what follows, we only work with the matrix $N$, and our goal is to prove that $\gamma_2(N)=\Omega(\sqrt{d_0})$, which then implies $\gamma_2(M)=\Omega(\sqrt{d})$. To simplify notation, we write $d$ instead of $d_0$.

Let the size of $N$ be $m\times n$, and recall that for any choice of $u\in \mathbb{R}^m$ and $v\in \mathbb{R}^n$ with $||u||_2=||v||_2=1$, we have
$$\gamma_2(N)\geq ||M\circ (uv^T)||_{\tr}.$$
Let $f$ be the number of one entries of $N$, and for $i=1,\dots,m$, let $d_i$ be the number of one entries of row $i$. Note that $f=d_1+\dots+d_m$. Let $u\in \mathbb{R}^m$ be defined as $u(i)=\sqrt{d_i/f}$ for $i\in [m]$, and let $v\in \mathbb{R}^n$ be defined as $v(i)=1/\sqrt{n}$. Then $||u||_2=||v||_2=1$. Let $A=N\circ (uv^T)$, then $\gamma_2(N)\geq ||A||_{\tr}$, so it is enough to prove that $||A||_{\tr}=\Omega(\sqrt{d})$. Let $\sigma_1\geq \dots\geq \sigma_m\geq 0$ be the singular values of $A$ (with possibly zeros added to get exactly $m$ of them), then $\sigma_1^2,\dots,\sigma_m^2$ are the eigenvalues of $B=AA^T$. 

Define the auxiliary graph $H$ on $[m]$, where for $i,i'\in [m]$, $i\neq i'$, we have $i\sim i'$ in $H$ if there is some index $j\in [n]$ such that $N(i,j)=N(i',j)=1$. As $M$ is four cycle-free, there is at most one such index $j$ for every pair $(i,i')$. With this notation, we can write
$$B(i,i')=\frac{1}{fn}\begin{cases} d_i^2 &\mbox{ if }i=i'\\
                        \sqrt{d_id_{i'}} &\mbox{ if }i\sim i'\\
                        0 &\mbox{ otherwise.}\end{cases}$$
For $t=1,\dots,\lceil\log_3 n\rceil=:p$, let $I_t\subset [m]$ be the set of indices $i$ such that $3^{t-1}\leq d_i\leq 3^t$. Then $I_1,\dots,I_p$ forms a partition of $[m]$, and we note that $I_t$ is empty if $t\leq \log_3 d-1$. Let $B_t=B[I_t\times I_t]$, then $B_t$ is a principal submatrix of $B$.

\begin{claim}
At least $\Omega(|I_t|)$ eigenvalues of $B_t$ are at least $\Omega(3^{2t}/(fn))$. 
\end{claim}

\begin{proof}
Let $s=|I_t|$, $D=3^{t-1}$, and let $\lambda_1\geq \dots\geq \lambda_s\geq 0$ be the eigenvalues of $B_t$. Then 
\begin{equation}\label{equ:trace}
    \lambda_1+\dots+\lambda_s=\tr(B_t)=\frac{1}{fn}\sum_{i\in I_t}d_i^2\geq \frac{sD^2}{fn},
\end{equation}
and
$$\lambda_1^2+\dots+\lambda_s^2=||B_t||_2^2.$$
Here,
$$||B_t||_2^2=\sum_{i,i'\in I_t}B(i,i')^2=\frac{1}{(fn)^2}\left[\sum_{i\in I_t}d_i^4+\sum_{i\sim i',i,i'\in I_t}d_id_{i'} \right]\leq \frac{1}{(fn)^2}\left[81sD^{4}+18e(H[I_t])D^{2}\right],$$
where $e(H[I_t])$ denotes the number of edges of the subgraph of $H$ induced on the vertex set $I_t$. Let $G$ be the bipartite graph, whose bi-adjacency matrix is $N[I_t\times [n]]$. Then $e(H[I_t])$ is the number of pairs $\{i,i'\}\in I_t^{(2)}$ such that $i$ and $i'$ has a common neighbour in $G$. As every column of $N$ has at most $6d$ one entries, every vertex in $[n]$ has degree at most $6d$ in $G$. Thus, for each $i\in I_t$, there are at most $6dd_i\leq 18dD$ vertices $i'\in I_t$ which have a common neighbour with $i$. Therefore, $e(H[I_t])\leq 12dDs\leq 36D^2s$, where we used in the last inequality that $s=0$ unless $t\geq \log_3 d-1$. In conclusion, we proved that 
$$\lambda_1^2+\dots+\lambda_s^2\leq \frac{1000sD^4}{(fn)^2}.$$
Let $C=2000$, and let $r\leq s$ be the largest index such that $\lambda_r\geq \frac{CD^2}{fn}$. By the previous inequality, we have $r\leq \frac{1000s}{C^2}$. But then by the inequality between the arithmetic and square mean, $$\lambda_1+\dots+\lambda_r\leq r^{1/2}(\lambda_1^2+\dots+\lambda_r^2)^{1/2}\leq r^{1/2}\left(\frac{1000sD^4}{(fn)^2}\right)^{1/2}\leq \frac{sD^2}{2fn}.$$
Hence, comparing this with (\ref{equ:trace}), we deduce that
$$\lambda_{r+1}+\dots+\lambda_{s}\geq \frac{sD^2}{2fn}.$$
As $\frac{CD^2}{fn}\geq \lambda_{r+1}\geq\dots\geq \lambda_s$, this is only possible if at least $\frac{s}{4C}$ among $\lambda_{r+1},\dots,\lambda_s$ is at least $\frac{D^2}{4fn}$. This finishes the proof.
\end{proof}
As $B_t$ is a principal submatrix of $B$, its eigenvalues interlace the eigenvalues of $B$. Therefore, the previous claim implies that at least $c|I_t|$ eigenvalues of $B$ are at least $c3^{2t}/(fn)$ for some absolute constant $c>0$, and thus at least $c|I_t|$ singular values of $A$ are at least $c3^t/\sqrt{fn}$. 

We are almost done. Note that 
$$f=\sum_{i=1}^md_i\leq \sum_{t=1}^p 3^{t+1}|I_t|.$$
In order to bound $||A||_{\tr}=\sigma_1+\dots+\sigma_m$, we observe that for every $t$, if $|I_t|\geq \max\{|I_{t+1}|,\dots,|I_p|\}=:z_t$, then $$\sum_{i=z_t+1}^{|I_t|}\sigma_i\geq \frac{c3^t}{\sqrt{fn}}\cdot (c|I_t|-c|I_{t+1}|-\dots-c|I_p|).$$ Hence, 
\begin{align*}
||A||_{\tr}&=\sum_{i=1}^m \sigma_i\geq \sum_{t=1}^p \frac{c3^t}{\sqrt{fn}}\cdot (c|I_t|-c|I_{t+1}|-c|I_2|-\dots-c|I_{p}|)\\
&\geq \frac{c^2}{\sqrt{fn}}\sum_{t=1}^p |I_t|(3^t-3^{t-1}-\dots-3-1) \geq \frac{c^2}{2\sqrt{fn}}\sum_{t=1}^p3^t |I_t|\geq \frac{c^2\sqrt{f}}{6\sqrt{n}}.
\end{align*}
Here, in the first inequality, we use that if $|I_t|< z_t$, then the contribution of the $t$-th term is negative in the sum anyway. Finally, recall that every column of $N$ contains at least $d/2$ one entries, so $f\geq dn/2$. Therefore, $||A||_{\tr}\geq \frac{c^2}{12}\sqrt{d}$, finishing the proof.
\end{proof}

\begin{proof}[Proof of Theorem \ref{thm:main}]
We have $\gamma_2(M)\leq 2\sqrt{\dgc(M)}$ by Claim \ref{claim:dgc}, so it remains to prove that $\gamma_2(M)=\Omega(\sqrt{\dgc(M)})$. Let $N$ be a submatrix of $M$ of minimum degree $d=\dgc(M)$, then the average degree of $N$ is at least $d$. Hence, the previous lemma implies the desired lower bound by noting that $\gamma_2(M)\geq \gamma_2(N)$. 
\end{proof}

Next, we prove Theorem \ref{thm:degree_bounded}. In the proof, we use the following graph theoretic result of Gir\~ao and Hunter \cite{GH}, which tells us that in order to show that a hereditary graph family is degree-bounded, it is enough to consider its  four cycle-free members.

\begin{theorem}[Theorem 1.3 in \cite{GH}]\label{thm:GH}
Let $k\geq 2$ and let $t$ be sufficiently large. Every graph with average degree at least $t^{5000k^4}$ either contains $K_{t,t}$, or it contains an induced four cycle-free subgraph with average degree at least $k$.
\end{theorem}

\begin{proof}[Proof of Theorem \ref{thm:degree_bounded}]
Let $M$ be a Boolean matrix with no $t\times t$ all-ones submatrix, and $\gamma_2(M)\leq \gamma$. Let $d$ be the average degree of $M$. It follows from Theorem \ref{thm:main} that if $N$ is a four cycle-free submatrix of $M$, then the average degree of $N$ is at most $k=O(\gamma^2)$. However, by Theorem \ref{thm:GH}, if $d>t^{5000k^4}$, then $N$ contains a four cycle-free submatrix of average degree at least $k$. Hence, $d<t^{O(\gamma^8)}$,  so choosing $d(\gamma,t)=t^{c\gamma^8}$ for some sufficiently large constant $c$ finishes the proof. 
\end{proof}

Finally, we prove Corollary \ref{cor:equivalence}.

\begin{proof}[Proof of Corollary \ref{cor:equivalence}]
The equivalence $2.\Leftrightarrow 3.$ follows from Claim \ref{claim:blocky}. The implication $2.\Rightarrow 1.$ follows from the inequality $\gamma_2(M)\leq \block(M)$. It remains to show that $1.\Rightarrow 3.$. Let $M\in\mathcal{M}$ such that $\gamma_2(M)\leq \gamma$. If $N$ is a submatrix of $M$, then $\gamma_2(N)\leq \gamma$ by monotonicity. Hence, applying Theorem \ref{thm:degree_bounded} to $N$, we get that the average degree of $N$ is at most $d$ for some $d=d(\gamma,t)$. But then $N$ has a row or column with at most $d$ one entries. As this is true for every submatrix $N$, the degeneracy of $M$ is at most $d$ as well.
\end{proof}

\section{Point-polytope incidences}\label{sect:zara}

In this section, we prove Theorems \ref{thm:zara1} and \ref{thm:zara2}. But first, we provide a formal definition of semilinear graphs, following \cite{BCSTT}.

\begin{definition}[Semilinear graphs]
A graph $G$ is \emph{semilinear} of complexity $(s,u)$ of dimension $(d_1,d_2)$ if $V(G)=V_1\cup V_2$ with $V_1\subset \mathbb{R}^{d_1}$, $V_2\subset \mathbb{R}^{d_2}$, and there exist $su$ linear functions $f_{i,j}:\mathbb{R}^{d_1+d_2}\rightarrow \mathbb{R}$ for $(i,j)\in [s]\times [u]$ such that for every $(x,y)\in V_1\times V_2$, $\{x,y\}$ is an edge if and only if 
$$\exists j\in [u], \forall i\in [s]: f_{i,j}(x,y)<0.$$ 
In case $d_1=d_2=d$, the dimension of $G$ is simply $d$.
\end{definition}

The main result of this section is the following theorem, which immediately implies both Theorems \ref{thm:zara1} and \ref{thm:zara2}.

\begin{theorem}\label{thm:main_zara}
Let $G$ be a semilinear graph of complexity $(s,u)$ of dimension $(d_1,d_2)$. If $G$ is $K_{t,t}$-free, then the average degree of $G$ is at most $O_{d_1,s,u}(t(\log n)^{4d_1+2}(\log\log n)^{s})$. 
\end{theorem}

First, we present a weaker upper bound of the form $O_{s,u}(t(\log (n/t))^{s-1})$, which is used as a ''boosting'' step. This is similar to the upper bound $O_{s,u}(t(\log n)^{s})$ proved in \cite{BCSTT}. However, the crucial difference is the dependence on $t$, as the former gives much stronger bounds when $t$ is close to $n$. This improvement is important to achieve optimal dependence on $t$ in Theorem~\ref{thm:main_zara}. Our proof follows the simple divide-and-conquer approach of \cite{BCSTT} and \cite{TZ}.

\begin{lemma}\label{lemma:weak_upper}
    Let $G$ be a semilinear graph on $n$ vertices of complexity $(s,u)$. If $G$ contains no $K_{t,t}$, then the average degree of $G$ is at most $O_{s,u}(t(\log (n/t))^{s-1})$.
\end{lemma}

\begin{proof}
By the definition of semilinear graphs, there exist $u$ semilinear graphs $G_1,\dots,G_u$ of complexity $(s,1)$ on vertex set $V(G)$, whose union is $G$. Hence, it is enough to prove that $e(G_i)=O_{s}(t(\log (n/t))^{s-1})$.

To simplify notation, we assume that $G$ is semilinear of complexity $(s,1)$, and write $V=V_1\cup V_2$. Then there exist $s$ linear functions $f_1,\dots,f_s$ such that $(x,y)\in V_1\times V_2$ is an edge if and only if $f_{i}(x,y)<0$. As $f_{i}$ is linear, we can write $f_i(x,y)=g_i(x)+h_i(y)$. For every $x\in V_1$, let $\tilde{x}=(g_i(x))_{i\in [s]}\in \mathbb{R}^s$, and for every $y\in V_2$, let $\tilde{y}=(-h_i(y))_{i\in [s]}\in\mathbb{R}^s$. Then $(x,y)$ is an edge if and only if $\tilde{x}\prec \tilde{y}$, where $\prec$ denotes the usual coordinate-wise ordering (i.e. $(a_1,\dots,a_s)\prec (b_1,\dots,b_s)$ if $a_i<b_i$ for every $i\in [s]$). Let $U_i=\{\tilde{x}: x\in V_i\}$ for $i=1,2$, and consider $G$ as the graph on vertex set $U_1\cup U_2$. 

 Let $f_{s}(n)$ denote the maximum number of edges of a $K_{t,t}$-free $n$ vertex graph defined in the manner above. We aim to show that $f_s(n)=O_s(tn(\log (n/t)^{s-1}))$.  We proceed by induction on $s$ and $n$. Consider the base case $s=1$. In this case, $U_1\cup U_2\subset \mathbb{R}$, and $\prec$ is the usual ordering of real numbers. Delete the $t$ largest elements of $U_1$, and the $t$ smallest elements of $U_2$. Then we deleted at most $2tn$ edges of $G$. If $G$ still has as an edge $(x,y)$, then $x<y$, and the $2t$ deleted vertices form a copy of $K_{t,t}$. Therefore, we must have $e(G)\leq 2tn$, confirming the case $s=1$.

Now assume that $s\geq 2$, then $V(G)\subset \mathbb{R}^s$. We use the trivial bound $f_s(n)\leq n^2$ in case $n\leq t$. Therefore, $f_s(n)\leq tn(\log (n/t))^{s-1}$ is satisfied if $s=1$ or $n\leq t$. Now consider $n>t$. Let $H$ be a hyperplane orthogonal to the last coordinate axis in $\mathbb{R}^s$ such that the two half-spaces bounded by $H$ both contain at most $\lceil n/2\rceil$ points of $V(G)$. Let $A\cup B$ be the partition of $V(G)$ given by $H$ with the elements of $A$ having smaller last coordinate. Then $G[A]$ and $G[B]$ both have at most $f_s(\lceil n/2\rceil)$ edges. Furthermore, we can count the number of edges between $A$ and $B$ as follows. Let $W_1$ be the projection of $U_1\cap A$ to $H$, and let $W_2$ be the projection of $U_2\cap B$ to $H$. Then we can view $W_1\cup W_2$ as a subset of $\mathbb{R}^{s-1}$, and for $x\in U_1\cap A$ and $y\in U_2\cap B$, we have $x\prec y$ if and only if $x'\prec y'$, where $x'$ are $y'$ are the projections of $x$ and $y$. Therefore, the number of edges between $A$ and $B$ is bounded by $f_{s-1}(n)$. In conclusion, we get that 
$$f_s(n)\leq 2f_s(\lceil n/2\rceil)+f_{s-1}(n).$$
It is easy to show that with the induction hypothesis $f_{s-1}(n)\leq O_s(tn(\log (n/t))^{s-2})$ and boundary condition $f_{s}(n)\leq n^2$ if $n<t$, we get that $f_s(n)=O_s(tn((\log n/t)^{s-1}))$.
 \end{proof}

Recall that if $\mathcal{H}$ is a set of half-spaces in $\mathbb{R}^d$, then  $\POL(\mathcal{H})$ denotes the set of all polytopes that can be written as $\bigcap_{H\in \mathcal{H}} H'$, where $H'$ is some translation of $H$. The next result, due to Nikolov \cite{Nikolov}, is one of the key ingredients in our proof. See the remark after Theorem 12 in \cite{Nikolov} for the following theorem, which we use as a black box.
 
\begin{theorem}\label{thm:nikolov}
Let $\mathcal{H}$ be a set of $D$  half-spaces in $\mathbb{R}^d$, and let $P$ be a set of $n$ points. If $M$ is the incidence matrix of $P$ and $\POL(\mathcal{H})$, then $\gamma_2(M)=O_{d,D}((\log n)^{d})$. 
\end{theorem}

From this, we conclude the following bound on the $\gamma_2$-norm of semilinear graphs.

\begin{lemma}\label{lemma:semilin}
Let $G$ be a semilinear graph on $n$ vertices of complexity $(s,u)$ of dimension $(d_1,d_2)$. If $M$ is the bi-adjacency matrix of $G$, then $\gamma_2(M)=O_{d_1,s,u}((\log n)^{d_1})$.
\end{lemma}

\begin{proof}
    Let $V(G)=V_1\cup V_2$ and let $f_{i,j}:\mathbb{R}^{d_1+d_2}\rightarrow \mathbb{R}$ be the defining linear functions of $G$. Then $f_{i,j}(x,y)=\langle a_{i,j},x\rangle+\langle b_{i,j},y\rangle +c_{i,j}$ with some $a_{i,j}\in \mathbb{R}^{d_1}, b_{i,j}\in \mathbb{R}^{d_2}, c_{i,j}\in \mathbb{R}$.
    For $y\in V_2$ and $\varepsilon\in \{-1,1\}^{[s]\times [t]}$, let $Q_{\varepsilon}(y)$ be the polytope defined as
    $$Q_{\varepsilon}(y)=\bigcap_{(i,j)\in [s]\times [t]}\{ x\in \mathbb{R}^{d_1}: \varepsilon_{i,j}f_{i,j}(x,y)<0\}.$$
    Then for every $y$, the $2^{su}$ polytopes $Q_{\varepsilon}(y)$ are pairwise disjoint. Also, with an appropriate choice of $E\subset \{-1,1\}^{[s]\times [u]}$, we have that $\{x,y\}$ is an edge of $G$ if and only if $x\in \bigcup_{\varepsilon\in E} Q_{\varepsilon}(y)$.

    Let $M_{\varepsilon}$ be the incidence matrix of $V_1$ and $\{Q_{\varepsilon}(y)\}_{y\in Y}$. If $\mathcal{H}_{\varepsilon}$ is the set of half-spaces $\{\varepsilon_{i,j}\langle a_{i,j},x\rangle<0\}$, then $M_{\varepsilon}$ is a submatrix of the incidence matrix of $V_1$ and $\POL(\mathcal{H}_{\varepsilon})$. Therefore, $\gamma_2(M_{\varepsilon})=O_{d_1,s,u}((\log n)^{d_1})$ by Theorem \ref{thm:nikolov} and the monotonicity of the $\gamma_2$-norm. As $M=\sum_{\varepsilon\in E}M_{\varepsilon}$, we conclude that $\gamma_2(M)=O_{d_1,s,u}((\log n)^{d_1})$ as well by the subadditivity of the $\gamma_2$-norm.
\end{proof}

In the case of $K_{2,2}$-free semilinear graphs, we get the following immediate corollary of the previous lemma and Theorem \ref{thm:main}.

\begin{theorem}
Let $G$ be a semilinear graph on $n$ vertices of complexity $(s,u)$ of dimension $(d_1,d_2)$. If $G$ is four cycle-free, then the average degree of $G$ is $O_{d_1,s,u}((\log n)^{2d_1})$.
\end{theorem}

\begin{proof}
Let $d$ be the average degree of $G$ and let $M$ be the bi-adjacency matrix of $G$. Then $\dgc(M)\geq d/2$, and thus by Theorem \ref{thm:main}, $\gamma_2(M)=\Omega(\sqrt{d})$. On the other hand, by Lemma \ref{lemma:semilin}, we also have $\gamma_2(M)=O_{d_1,s,u}((\log n)^{d_1})$. Therefore, we conclude that $d=O_{d_1,s,u}((\log n)^{2d_1})$
\end{proof}

In order to prove Theorem \ref{thm:main_zara} for any $t$, we have to work harder. The main idea is as follows. We consider the bi-adjacency matrix $M$ of the graph $G$, and using that it has small $\gamma_2$-norm, we show that $G$ contains a fairly dense subgraph $G'$. Then, we apply Lemma \ref{lemma:weak_upper} to $G'$. In order to find the dense subgraph, we first show that $G$ contains many four cycles, using spectral properties of $M$. Then, we argue that this is only possible if there is a group of vertices of $G$, whose neighborhoods highly overlap. Thus, picking such a group together with its neighborhood forms a dense subgraph. The core of this argument is the next lemma, which can be also found in \cite{CHHNPS}. However, due to its simplicity, we present a proof as well.

\begin{lemma}\label{lemma:trace_bound}
Let $M\in \mathbb{R}^{m\times n}$ be non-zero. Then
$$||M||_4^4\geq \frac{||M||_2^6}{mn\gamma_2(M)^2}.$$ 
\end{lemma}

\begin{proof}
We use the following generalization of H\"{o}lder's inequality.

\begin{lemma}[Generalized H\"older's inequality]
Let $x\in \mathbb{R}^n$, and let $||x||_p=(|x(1)|^p+\dots+|x(n)|^p)^{1/p}$ denote the $p$-norm of $x$. If $p_1,\dots,p_k,r>0$ such that $\frac{1}{r}=\frac{1}{p_1}+\dots+\frac{1}{p_k}$, then $||x||_{p_1}\dots ||x||_{p_k}\geq ||x^{k}||_r $, where $x^{k}$ is the vector defined as $(x^k)(i)=(x(i))^k$.
\end{lemma}

Let $\sigma$ be the vector of singular values of $M$. Then $||\sigma||_p=||M||_p$ for any $p>0$ and $||\sigma^{k}||_r=||M||_{kr}^{k}$ for any $k,r>0$. Hence, applying H\"older's inequality  with the parameters $k=3,p_1=1,p_2=p_3=4,r=2/3$,  we arrive to the inequality
\begin{equation}\label{equ:holder}
    ||M||_{1} ||M||_4^{2}\geq ||M||_2^3.
\end{equation}
Next, we use the inequality $\gamma_2(M)\geq \frac{1}{\sqrt{mn}}||M||_{1}$ to bound $||M||_1$. From this, we get
$$||M||_4^2\geq \frac{||M||_2^3}{||M||_1}\geq \frac{||M||_2^3}{\sqrt{mn}\gamma_2(M)^2}.$$ 
This finishes the proof.
\end{proof}

Here, $||M||_2^2$ is the sum of the squares of the entries of $M$, and $||M||_4^4$ is the sum of the squares of entries of $M^TM$. In particular, if $M$ is a Boolean matrix, then $||M||_2^2$ is the number of one entries of $M$, and $||M||_4^4$ counts the number of homomorphic copies of four cycles in the corresponding bipartite graph. Before we apply the previous lemma to our matrix $M$, we do some regularization.

Say that an $m\times n$ Boolean matrix $M$ is \emph{$(p,q,d)$-biregular} if there exist integers $a,b>d$ such that every row of $M$ contains at most $a$ one entries, every column contains at most $b$ one entries, and the number of one entries of $M$ is at least $\max\{pam,qbn\}$.

\begin{lemma}\label{lemma:biregular1}
Let $M$ be an $m\times n$ Boolean matrix with average degree $d$. Then $M$ contains a submatrix $M_0$ such that either $M_0$ or $M_0^T$ is $(\frac{1}{2},\frac{1}{12\log_2 (m+n)},\frac{d}{2})$-biregular.
\end{lemma}

\begin{proof}
Let $A$ and $B$ be the vertex classes of $G$. First, we apply Lemma \ref{lemma:regularize} to find an $m_0\times n_0$ sized submatrix $M_0$ such that the degree of $M_0$ is $d'\geq d/3$, its minimum degree is at least $d'/2$, and, without loss of generality, every row of  $M_0$ has at most $a=6d'$ one entries. By a standard dyadic pigeon-hole argument, we can find a positive real number $b\geq d'$ and a subset of columns such that the submatrix $M'$ of $M_0$ formed by these columns contains at least $\frac{||M_0||_2^2}{\log_2 n}\geq \frac{m_0a}{12\log_2 n}$ one entries, and  $b/2<||M'||_{\col}^2\leq b$. If $M'$ has $m'$ columns, then $||M'||_2^2\geq m'b/2$, so $M'$ is $(\frac{1}{12\log_2 n},\frac{1}{2},\frac{d}{2})$-biregular.
\end{proof}

\begin{lemma}\label{lemma:biregular2}
Let $M$ be an $m\times n$ sized $(p,q,d)$-biregular matrix, and let $\alpha=\frac{p^2q}{2\gamma_2(M)^2}$. Then for every $z<\alpha d$,  $M$ contains a $z\times z$ sized submatrix with average degree at least $\alpha z$.
\end{lemma}

\begin{proof}
Let $a,b\geq d$ such that  every row of $M$ contains at most $a$ one entries, every column contains at most $b$ one entries, and the number of one entries of $M$ is at least $\max\{pam,qbn\}$. By Lemma~\ref{lemma:trace_bound}, 
$$||M||_4^4\geq \frac{||M||_2^6}{mn\gamma_2(M)^2}.$$
Here, $||M||_2^2\geq \max\{pam,qbn\}$, so the right-hand-side can be lower bounded by 
$$\frac{p^2qa^2bm}{\gamma_2(M)^2}.$$
Write $N=||M||_4^4=\tr((MM^T)^2)$, then $N$ is the number of 4-tuples $(i,i',j,j')\in [m]^2\times [n]^2$ such that $M(i,j)=M(i,j')=M(i',j)=M(i',j')=1$. Call such a 4-tuple a \emph{square}. There exists an index $i_0\in [m]$ such that $i_0$ is the first coordinate in at least $\frac{N}{m}$ squares. Let $J\subset [n]$ be the set of indices $j$ such that $M(i_0,j)=1$, and let $M'=M[[m]\times J]$. Let $s_1,\dots,s_m$ be the number of one entries in the rows of $M'$, then the number of squares in which $i_0$ is the first coordinate is $s_1^2+\dots+s_m^2$, so 
\begin{equation}\label{equ:x1}
    s_1^2+\dots+s_m^2\geq \frac{N}{m}\geq \frac{p^2qa^2b}{\gamma_2(M)^2}. 
\end{equation}
On the other hand, we know that $|J|\leq a$, and as each column of $M'$ contains at most $b$ one entries, we also have $s_1+\dots+s_m\leq |J|b\leq ab$. Let $x=\frac{p^2qa}{2\gamma_2(M)^2}$, and assume that there are $t$ numbers among $s_1,\dots,s_m$ that are larger than $x$. Then
\begin{equation}\label{equ:x2}
s_1^2+\dots+s_m^2\leq t|J|^2+x(s_1+\dots+s_m)\leq ta^2+xab.
\end{equation}
Comparing (\ref{equ:x1}) and (\ref{equ:x2}), we get 
$$t\geq \frac{p^2qb}{2\gamma_2(M)^2}.$$
Let $M''$ be the submatrix of $M'$, where each row contains at least $x$ one entries. Then  $M''$ has $t$ rows, and at least $tx$ one entries. Finally, let $z\leq \frac{p^2qd}{2\gamma_2(M)^2}=\alpha d$. Then $t\geq z$ and $|J|\geq x\geq z$. Let $M_0$ be a random $z\times z$ submatrix of $M''$, chosen from the uniform distribution. Then the expected number of one entries of $M_0$ is at least
$$tx\cdot \frac{z}{t}\cdot\frac{z}{|J|}=z^2\frac{x}{|J|}\geq \frac{p^2q z^2}{2\gamma_2(M)^2}.$$
Therefore, there exists a choice for the $z\times z$ matrix $M_0$ such that the average degree of $M_0$ is at least $\frac{p^2q z}{2\gamma_2(M)^2}=\alpha z$, finishing the proof.
\end{proof}

Combining the previous two lemmas, we get the following corollary.

\begin{lemma}\label{lemma:dense_submatrix}
Let $M$ be an $m\times n$ Boolean matrix of average degree $d$, and let $$\alpha=(200\gamma_2(M)^2\log_2(m+n))^{-1}.$$ Then for every $z\leq \alpha d$, $M$ contains a $z\times z$ submatrix of average degree at least $\alpha z$. 
\end{lemma}

\begin{proof}
By Lemma \ref{lemma:biregular1}, $M$ contains a submatrix $M_0$ such that either $M_0$ or $M_0^T$ is $(\frac{1}{2},\frac{1}{12\log_2 (m+n)},\frac{d}{2})$-biregular. Without loss of generality, we may assume that the first case happens. Then applying Lemma \ref{lemma:biregular2} to $M_0$ gives the desired submatrix.
\end{proof}

We are ready to prove the main theorem of this section.

\begin{proof}[Proof of Theorem \ref{thm:main_zara}]
Let $D$ be the average degree of $G$, and let $M$ be its bi-adjacency matrix. Assume that $D\geq Ct(\log n)^{4d_1+2}(\log\log n)^{s}$, where $C$ is a sufficiently large  constant only depending on $d_1,s,u$. By Lemma \ref{lemma:semilin}, $\gamma_2(M)=O_{d_1,s,u}((\log n)^{d_1})$. Let $$\alpha=\frac{1}{400\gamma_2(M)^2\log_2(n)}=\Omega_{d_1,s,u}((\log n)^{-2d_1-1}),$$ then by Lemma \ref{lemma:dense_submatrix}, $M$ contains a $z\times z$ submatrix $M'$ with $$z=\alpha D=\Omega_{d_1,s,u}(Ct(\log n)^{2d_1+1}(\log\log n)^s))$$ such that the average degree of $M'$ is at least $\alpha z=\Omega_{d_1,s,u}(Ct(\log\log n)^s)$. Here, $M'$ is the bi-adjacency matrix of a semilinear graph of complexity $(s,u)$ on $2z$ vertices with no $K_{t,t}$. Hence,  by Lemma \ref{lemma:weak_upper}, its average degree is at most 
\begin{align*}
    D_0&=O_{s,u}(t(\log (z/t))^{s})=O_{d_1,s,u}\left(t\left[\log C (\log n)^{2d_1+1}(\log \log n)^s\right]^s\right)\\
    &=O_{d_1,s,u}(t\log C+t(\log\log n)^s).
\end{align*}
 By choosing $C$ sufficiently large, we get that $\alpha z>D_0$, contradiction.
\end{proof}

\section{Concluding remarks}

We proved that for every $\gamma>1$ and integer $t\geq 2$, there exists $d(\gamma,t)$ such that every Boolean matrix with no $t\times t$ all-ones submatrix and $\gamma_2$-norm at most $\gamma$ has average degree at most $d(\gamma,t)$. It would be interesting to understand how  $d(\gamma,t)$ depends on the parameters $\gamma$ and $t$.
\begin{problem}
How does $d(\gamma,t)$ depend on $\gamma$ and $t$?
\end{problem}

For $t=2$, we proved the sharp bound $d(\gamma,2)=O(\gamma^2)$. However, for $t>2$, our proof only implies $d(\gamma,t)=t^{O(\gamma^8)}$. On the other, we believe that $d(\gamma,t)$ should grow linearly in $t$, which is also closely related to  Conjecture I in \cite{HHH}. In contrast, we prove that $d(\gamma,t)$ grows at least exponentially in $\gamma$.

\begin{lemma}
Let $\gamma>4$ and $n$ be sufficiently large with respect to $\gamma$. Then there exists an $n\times n$ Boolean matrix $M$ with $(1-o(1))n^2$ one entries, $\gamma_2(M)\leq \gamma$, such that $M$  contains no $t\times t$ all-ones submatrix for $t>4\cdot 2^{-\gamma}n$. 

In particular, for every $\gamma>1$ and every $t>t_0(\gamma)$, $d(\gamma,t)=\Omega(2^{\gamma}t)$.
\end{lemma}

\begin{proof}
 Let $\ell=\lfloor\gamma-1\rfloor>2$, and let $m$ be an integer sufficiently large with respect to $\ell$. Let $S$ be an $m$ element ground set, let $p=m^{3/2-\ell}$, and let $\mathcal{F}$ be a random sample of the $\ell$-element subsets of $S$, where each $\ell$-element set is included independently with probability $p$. Let $X=|\mathcal{F}|$, then $\mathbb{E}(X)=p\binom{m}{\ell}=\Omega_{\ell}(m^{3/2})$, and by standard concentration arguments, $\mathbb{P}(X>\mathbb{E}(X)/2)>0.9$. Let $Y$ be the number of pairs of sets in $\mathcal{F}$, whose intersection has size at least two. Then $\mathbb{E}(Y)<p^2m^{2\ell-2}=m$, so by Markov's inequality, $\mathbb{P}(Y<10m)\geq 0.9$. Furthermore, let $Y'$ be the number of pairs of sets in $\mathcal{F}$, whose intersection has size exactly 1. Then $\mathbb{E}(Y')\leq p^2 m^{2\ell-1}=m^2$, so by Markov's inequality, $\mathbb{P}(Y'<10m^2)\geq 0.9$.  Finally, let $T\subset S$ be any set of size $m/2$, and let $Z_T$ be the number of elements of $\mathcal{F}$ completely contained in $T$. Then $\mathbb{E}(Z_T)=p\binom{m/2}{\ell}<2^{-\ell}\mathbb{E}(X)$. By the multiplicative Chernoff inequality, we can write
$$\mathbb{P}(Z_T\geq 2\mathbb{E}(Z_T))\leq \exp\left(-\frac{1}{3}\mathbb{E}(Z_T)\right)\leq \exp(-\Omega_{\ell}(m^{3/2})).$$
Hence, as the number of $m/2$ element subsets of $S$ is at most $2^m$, a simple application of the union bound  implies $$\mathbb{P}(\forall T\subset S, |T|=m/2: Z_T\leq 2\mathbb{E}(Z_T))>0.9.$$
In conclusion, there exists a choice for $\mathcal{F}$ such that $X>\mathbb{E}(X)/2$, $Y<10m$, $Y'<10m^2$, and $Z_T\leq 2\mathbb{E}(Z_T)\leq 2\cdot 2^{-\ell}\mathbb{E}(X)$ for every $m/2$ element set $T$. For each  pair of sets intersecting in more than one element in $\mathcal{F}$, remove one of them from $\mathcal{F}$, and let $\mathcal{F}'$ be the resulting set. Let $n=|\mathcal{F}'|/2=\Omega_{\ell}(m^{3/2})$, then $n>X/2-5m\geq \mathbb{E}(X)/4=\Omega_{\ell}(m^{3/2})$, and thus $Z_T\leq 8\cdot 2^{-\ell}n$ for every $T$.

Define the $n\times n$ matrix $M_0$ as follows. Let $\mathcal{A}\cup\mathcal{B}$ be an arbitrary partition of $\mathcal{F}'$ into two $n$ element sets. Let $U$ be the $n\times m$ matrix, whose rows are the characteristic vectors of the elements of $\mathcal{A}$, let $V$ be the $m\times n$ matrix, whose columns are the  characteristic vectors of the elements of $\mathcal{B}$, and set $M_0=UV$. As each row of $U$ and each column of $V$ is a zero-one vector with $\ell$ one entries, we have $\gamma_2(M_0)\leq ||U||_{\row}||V||_{\col}=\ell$. Also, $M_0$ is a Boolean matrix, which is guaranteed by the fact that any two distinct sets in $\mathcal{F}'$ intersect in 0 or 1 elements. The number of one entries of $M_0$ is at most $Y'<10m^2=o(n^2)$. Finally, $M_0$ contains no $t\times t$ all-zeros submatrix if $t>8\cdot 2^{-\ell}n$. Indeed, a $t\times t$ all-zeros submatrix corresponds to subfamilies $\mathcal{A}'\subset \mathcal{A}$ and $\mathcal{B}'\subset \mathcal{B}$ of sizes $t$ such that every element of $\mathcal{A}'$ is disjoint from every element of $\mathcal{B}'$. In other words, if $T_1=\bigcup_{A\in\mathcal{A}'}A$  and $T_2=\bigcup_{B\in\mathcal{B}'} B$, then $T_1$ and $T_2$ are disjoint. But then at least one of $T_1$ or $T_2$ has size at most $m/2$, without loss of generality, $|T_1|\leq m/2$. The number of elements of $\mathcal{F}'$ contained in $T_1$ is at most $8\cdot 2^{-\ell}n$, so we indeed have $t\leq 8\cdot 2^{-\ell}n$.

In order to get our desired matrix $M$, we just take the complement of $M_0$, that is, $M=J-M_0$. Then $\gamma_2(M)\leq 1+\gamma_2(M_0)\leq \ell+1\leq \gamma$, and $M$ satisfies the desired properties.
\end{proof}

 \section*{Acknowledgments}
 We would like to thank Zach Hunter, Aleksa Milojevi\'c, and Benny Sudakov for many fruitful discussions, and Lianna Hambardzumyan for her comments  that helped greatly improving the paper. Furthermore, we would like to thank Shachar Lovett for pointing out that certain results about the discrepancy in the previous version of our paper were trivial.

\end{document}